\documentclass[a4paper,11pt]{article}

\usepackage{amsmath}
\usepackage{amssymb}
\usepackage{amsthm}
\usepackage{amsfonts}
\usepackage{mathtools}
\usepackage{ascmac}
\usepackage{spalign}
\usepackage{fancybox}
\usepackage[mathscr]{eucal}
\usepackage{footnpag}
\usepackage{siunitx}
\usepackage{multicol}
\usepackage{bbm}
\usepackage{framed}
\usepackage{here}

\theoremstyle{plain}
\newtheorem{theorem}{Theorem}[section]
\newtheorem*{theorem*}{Theorem}

\newtheorem*{definition*}{Definition}

\newtheorem*{proposition*}{Proposition}
\newtheorem{lemma}{Lemma}[section]
\newtheorem*{lemma*}{Lemma}
\newtheorem{corollary}{Corollary}[section]
\newtheorem*{corollary*}{Corollary}
\newtheorem{remark}{Remark}[section]
\newtheorem*{remark*}{Remark}

\title{Rooted spanning forests in wheel graphs: Fibonacci--Lucas formulas and extremal root configurations}

\author{
Shunya Tamura\thanks{
Corresponding author. 
Okegawa City Okegawa West Junior High School, 
Saitama, 363-0027, Japan, 
e-mail: shunya.tamura059@gmail.com
}
\and
Yuuho Tanaka\thanks
{Faculty of Science and Technology, Oita University, 
Oita, 870-1192, Japan, 
e-mail: tanaka-yuuho@oita-u.ac.jp
}
}

\date{}

\begin{document}
\maketitle
\begin{abstract}
In this paper, we study rooted spanning forests in the wheel graph $W_{N+1}$.
For a root set $R$ consisting of rim vertices, we give an explicit formula,
in terms of Fibonacci and Lucas numbers, for the number of rooted spanning forests
in which each connected component contains exactly one vertex of $R$.
When the central vertex is also included in the root set, we obtain a simple
product formula involving only Fibonacci numbers.

Furthermore, by using the correspondence between rooted spanning forests
and vertex identification, we derive explicit formulas for the number of
spanning trees of quotient graphs of wheel graphs obtained by identifying
several vertices into one vertex.
We then rewrite the obtained formulas combinatorially in terms of the numbers
of matchings in path graphs and cycle graphs.
In addition, for a fixed number $r$ of rim root vertices, we express the sum
of the numbers of rooted spanning forests over all rim root sets of size $r$
as coefficients of generating functions, and we solve the extremal problem
for the arrangement of roots.
\end{abstract}



\section{Introduction}

The numbers of spanning trees and spanning forests of a graph are
fundamental enumerative invariants that reflect the structure of the graph
through minors of its Laplacian matrix \cite{Chaiken1982,ChebotarevShamis1997}.
In particular, it is known that Fibonacci and Lucas numbers naturally appear
in the numbers of spanning trees of wheel graphs and fan graphs
\cite{Hilton1974,Myers1971,Sedlacek1969}.

Let $G$ be a finite connected graph and let $\varnothing\ne R\subseteq V(G)$.
A spanning forest in which each connected component contains exactly one
vertex of $R$ is called a rooted spanning forest with root set $R$.
By the all-minors Matrix--Tree Theorem, the number of rooted spanning forests
of a graph is equal to the determinant of the principal submatrix obtained
from the Laplacian matrix by deleting the rows and columns corresponding
to the root set \cite{Chaiken1982}.
Rooted spanning forests are also related to the Matrix--Forest Theorem
\cite{ChebotarevShamis1997}.
Connections between Fibonacci and Lucas numbers, and rooted spanning forests in path and cycle graphs, have also been studied
\cite{Chebotarev2008}.

Counting rooted spanning forests is also related to counting spanning trees
of graphs obtained by identifying several vertices.
For wheel graphs with two identified vertices, Tamura and Tanaka
\cite{TamuraTanaka2026} gave formulas in terms of Fibonacci and Lucas numbers.
Moreover, Miezaki and Tamura \cite{MiezakiTamura2025} studied two-component
spanning forests separating two prescribed vertices in wheel graphs.
In this paper, we extend the result for two identified vertices to the case
of an arbitrary number of rim vertices, and also study summation problems
and extremal problems concerning root configurations.
More precisely, for a root set $R$ consisting of rim vertices of the wheel
graph $W_{N+1}$ with $|R|\geq1$, we explicitly express the number of rooted
spanning forests in which each connected component contains exactly one
vertex of $R$ in terms of Fibonacci and Lucas numbers.

\begin{theorem}
\label{thm:rim-roots}
Let $N\ge 3$, and consider the wheel graph $W_{N+1}$ with $N$ rim vertices
$v_1,v_2,\ldots,v_N$.
Let
\[
R=\{v_{a_1},v_{a_2},\ldots,v_{a_r}\},
\qquad
1\le a_1<a_2<\cdots<a_r\le N,
\]
and assume that $1\le r\le N$.
Define the cyclic intervals between adjacent root vertices by
\[
d_i=
\begin{cases}
a_{i+1}-a_i, & 1\le i\le r-1,\\
a_1+N-a_r, & i=r
\end{cases}. 
\]

The number $\mathcal F_{W_{N+1}}(R)$ of rooted spanning forests of
$W_{N+1}$ with root set $R$ consisting of rim vertices is given by
\[
\begin{aligned}
\mathcal F_{W_{N+1}}(R)
=\left(\prod_{i=1}^{r}F_{2d_i}\right)\left(\sum_{i=1}^{r}\frac{L_{2d_i}-2}{F_{2d_i}}\right)
=\sum_{i=1}^{r}\left(L_{2d_i}-2\right)\prod_{\substack{1\le j\le r\\j\ne i}}F_{2d_j}.
\end{aligned}
\]
\end{theorem}

When $r=1$, this gives the classical formula for the number of spanning trees
of a wheel graph. When $r=2$, it agrees with the known formula for the case
of two identified vertices.
If the central vertex $c$ is included in the root set, then an even simpler
product formula holds.

\begin{theorem}
\label{thm:center-root}
Let $N\ge 3$, $1\le r\le N$, and let $R$ be a set consisting of $r$ rim vertices.
Let $d_1,\ldots,d_r$ be the corresponding cyclic intervals.
Then the number of rooted spanning forests of $W_{N+1}$ with root set
$\{c\}\cup R$ is
\[
\mathcal F_{W_{N+1}}(\{c\}\cup R)
=
\prod_{i=1}^{r}F_{2d_i}.
\]
\end{theorem}

We also fix the number $r$ of rim root vertices and compute the sum over
all rim root sets of size $r$.
Let $V_{\mathrm{rim}}$ be the set of rim vertices, and put
\[
A_{N,r}
=
\sum_{\substack{R\subseteq V_{\mathrm{rim}}\\|R|=r}}
\mathcal F_{W_{N+1}}(R),
\qquad
B_{N,r}
=
\sum_{\substack{R\subseteq V_{\mathrm{rim}}\\|R|=r}}
\mathcal F_{W_{N+1}}(\{c\}\cup R).
\]
Then
\[
A_{N,r}
=
N[z^{N-r}]
\frac{1+z}
{(1-z)(1-3z+z^2)^r}
\]
and
\[
B_{N,r}
=
\frac{N}{r}
[z^{N-r}]
\frac{1}
{(1-3z+z^2)^r}
\]
hold.

Another main result of this paper is the complete solution of an extremal
problem concerning the arrangement of roots.
Fix $N$ and $r$, and write
\[
N=qr+s,
\qquad
0\le s<r.
\]
Whether or not the central vertex is included in the root set, the number of
rooted spanning forests is maximized when, among the cyclic intervals
$d_1,\ldots,d_r$, exactly $s$ of them are equal to $q+1$ and the remaining
$r-s$ are equal to $q$.

On the other hand, the number of rooted spanning forests is minimized when
the cyclic interval sequence is a permutation of
\[
(N-r+1,1,\ldots,1).
\]

This paper is organized as follows.
In Section \ref{sec:preliminaries}, we prepare basic notation and facts
on rooted spanning forests.
In Section \ref{sec:rooted spanning forest}, we consider rooted spanning
forests both in the case where the root set consists only of rim vertices
and in the case where the central vertex is included in the root set, and
we derive explicit formulas for the number of rooted spanning forests.
We also discuss the relationship between rooted spanning forests and
matchings in path graphs and cycle graphs.
In Section \ref{sec:sum-formulas}, we derive summation formulas over all
rim root sets.
In Section \ref{sec:extremal}, we solve the maximization and minimization
problems concerning root configurations.
Finally, in Section \ref{sec:conclusion}, we give concluding remarks and
discuss future problems.

\section{Preliminaries}
\label{sec:preliminaries}

In this section, we collect notation and basic facts used in this paper.
Unless otherwise stated, all graphs considered are finite undirected graphs.
However, we allow multiple edges which arise from vertex identification.
Multiple edges are counted with their multiplicities, while loops are ignored
because they do not contribute to the number of spanning trees.
The sequence $\{F_n\}_{n\ge0}$ defined by
\[
F_0=0,\qquad F_1=1,\qquad
F_{n+2}=F_{n+1}+F_n
\qquad (n\ge0)
\]
is called the Fibonacci sequence.
Also, the sequence $\{L_n\}_{n\ge0}$ defined by
\[
L_0=2,\qquad L_1=1,\qquad
L_{n+2}=L_{n+1}+L_n
\qquad (n\ge0)
\]
is called the Lucas sequence.

\subsection{Spanning trees and rooted spanning forests}

We first recall spanning trees, rooted spanning forests, and their relation
to vertex identification.
For a finite connected graph $G$, we denote by $\tau(G)$ the number of
spanning trees of $G$.
A spanning forest of $G$ is a forest containing all vertices of $G$.
Let $\varnothing\ne R\subseteq V(G)$.
A rooted spanning forest with root set $R$ is a spanning forest in which
each connected component contains exactly one vertex of $R$.
In what follows, we denote the number of such forests by $\mathcal F_G(R)$.

Let
\[
L_G=D_G-A_G
\]
be the Laplacian matrix of $G$.
Here $D_G$ and $A_G$ denote the degree matrix and the adjacency matrix,
respectively.
If $G$ has multiple edges, then the entries of $A_G$ count edge multiplicities,
and the degrees are also counted with multiplicities.
For a vertex $v\in V(G)$, let $L_G[v]$ be the matrix obtained from $L_G$
by deleting the row and column corresponding to $v$.

\begin{theorem}[Kirchhoff's Matrix--Tree Theorem; see, e.g., \cite{Chaiken1982}]
\label{thm:matrix-tree}
For a finite connected graph $G$ and any vertex $v\in V(G)$, we have
\[
\tau(G)=\det L_G[v].
\]
\end{theorem}

Let $\varnothing\ne R\subseteq V(G)$.
We write $L_G[R]$ for the matrix obtained from $L_G$ by deleting all rows
and columns corresponding to the vertices in $R$.
When $R=V(G)$, the matrix $L_G[R]$ is the empty matrix, and its determinant
is defined to be $1$.

The following theorem is a principal-minor form of the all-minors
Matrix--Tree Theorem.

\begin{theorem}[All-minors Matrix--Tree Theorem;
Chaiken \cite{Chaiken1982}, Theorem 1]
\label{thm:all-minors}
Let $G$ be a finite connected graph, and let
$\varnothing\ne R\subseteq V(G)$.
Then
\[
\mathcal F_G(R)=\det L_G[R].
\]
\end{theorem}

In particular, when $|R|=1$, the number $\mathcal F_G(R)$ agrees with the
ordinary number of spanning trees $\tau(G)$.
Also, when $R=V(G)$, the only corresponding rooted spanning forest is the
forest with no edges, and hence $\mathcal F_G(V(G))=1$.

Next, we recall the relation between vertex identification and rooted
spanning forests.
Let $G$ be a finite connected graph, and let $G/R$ be the quotient graph
obtained by identifying all vertices in $R$ into a single vertex.
Multiple edges arising from this identification are kept with their
multiplicities, while loops are deleted.

\begin{theorem}
\label{thm:identification-forest}
Let $G$ be a finite connected graph, and let $\varnothing\ne R\subseteq V(G)$.
Then
\[
\tau(G/R)=\mathcal F_G(R).
\]
\end{theorem}

\begin{proof}
Let $G/R$ be the quotient graph obtained by identifying all vertices of
$R$ into one vertex $r_\ast$.

Consider the principal submatrix obtained from the Laplacian matrix of
$G/R$ by deleting the row and column corresponding to $r_\ast$.
By identifying the vertices of $R$, the adjacency relations among vertices
in $V(G)\setminus R$ do not change.
Moreover, for each vertex $v\in V(G)\setminus R$, the number of edges
from $v$ to $R$ is preserved as the multiplicity of the edges joining
$v$ and $r_\ast$ in the quotient graph.
Therefore, this principal submatrix coincides with $L_G[R]$.

Hence, by Kirchhoff's Matrix--Tree Theorem and the all-minors
Matrix--Tree Theorem, we obtain
\[
\tau(G/R)
=
\det L_G[R]
=
\mathcal F_G(R).
\]
\end{proof}

\subsection{Wheel graph}

Let $C_N$ be the cycle graph on $N$ vertices, where $N\ge3$, and write
\[
V(C_N)=\{v_1,v_2,\ldots,v_N\}.
\]
The indices are understood modulo $N$; in particular, $v_{N+1}=v_1$.

The wheel graph $W_{N+1}$ is the graph obtained from the cycle graph $C_N$
by adding a new vertex $c$ and joining $c$ to every $v_i$.
That is,
\[
\begin{aligned}
V(W_{N+1})
&=
\{c,v_1,v_2,\ldots,v_N\} \\
E(W_{N+1})
&=
\{\{v_i,v_{i+1}\}:1\le i\le N\}
\cup
\{\{c,v_i\}:1\le i\le N\}.
\end{aligned}
\]
We call $c$ the central vertex, and $v_1,\ldots,v_N$ the rim vertices.

The following classical formula is known for the number of spanning trees
of a wheel graph.

\begin{theorem}[Hilton \cite{Hilton1974},
Myers \cite{Myers1971},
Sedl\'a\v{c}ek \cite{Sedlacek1969}]
\label{thm:wheel-tree}
The number of spanning trees of the wheel graph $W_{N+1}$ is
\[
\tau(W_{N+1})=L_{2N}-2.
\]
\end{theorem}

The number of spanning trees of a fan graph is also expressed in terms
of Fibonacci numbers.
The fan graph $\operatorname{Fan}_{n+1}$ is the graph obtained from the
path $P_n$ by adding one new vertex and joining it to every vertex of
$P_n$.

\begin{theorem}[Hilton \cite{Hilton1974}]
\label{thm:fan-tree}
For every integer $n\ge 1$, we have
\[
\tau(\operatorname{Fan}_{n+1})=F_{2n}.
\]
\end{theorem}

A matching of a graph $G$ is a set of edges no two of which share an endpoint.
The empty set is also regarded as a matching.
A matching is called perfect if it covers all vertices of $G$.
Otherwise, it is called imperfect.

For $N\ge2$, let $m^*(C_{2N})$ denote the number of imperfect matchings
of the even cycle graph $C_{2N}$.
By Benjamin and Yerger \cite{BY2006}, we have
\[
m^*(C_{2N})=L_{2N}-2.
\]

In what follows, for convenience, we regard $C_2$ as a $2$-cycle with two
parallel edges between two vertices, and define
\[
m^*(C_2)=1=L_2-2.
\]

Finally, we define the cyclic intervals used in the following sections.
For the wheel graph $W_{N+1}$, consider a root set consisting of rim vertices,
\[
R=\{v_{a_1},v_{a_2},\ldots,v_{a_r}\}  \quad (1\leq a_1< a_2< \ldots < a_r\leq N,\;1\leq r\leq N).
\]

We define the cyclic interval $d_i$ between adjacent root vertices
$v_{a_i}$ and $v_{a_{i+1}}$ by
\[
d_i=
\begin{cases}
a_{i+1}-a_i, & 1\le i\le r-1, \\
a_1+N-a_r, & i=r.
\end{cases}
\]
Then each $d_i$ is a positive integer, and
\[
d_1+d_2+\cdots+d_r=N
\]
holds.

Changing the choice of the initial root vertex only cyclically permutes
the sequence of cyclic intervals $(d_1,d_2,\ldots,d_r)$.
Therefore, this cyclic interval sequence is uniquely determined up to
cyclic permutation.

\section{Rooted spanning forests in wheel graphs}
\label{sec:rooted spanning forest}

In this section, we describe formulas for the number of rooted spanning
forests in a wheel graph in terms of the cyclic interval sequence
$(d_1,d_2,\ldots,d_r)$.
We consider the case where the roots are rim vertices and the case where
the central vertex is included in the root set.
At the end of the section, we briefly discuss a relation with graph matchings.

\subsection{Rooted spanning forests with rim roots}
\label{sec:rim-roots}

In this subsection, we compute the number of rooted spanning forests whose
root set consists only of rim vertices.
The obtained formula also gives the number of spanning trees of the quotient
graph of a wheel graph obtained by identifying several rim vertices into one
vertex.
The main tools in this subsection are the all-minors Matrix--Tree Theorem
and the Schur complement.

We first prepare the determinant of a tridiagonal matrix which will be used
in the proof of the main theorem.
For $n\ge 1$, put
\[
A_n=
\begin{bmatrix}
3 & -1 & 0 & \cdots & 0\\
-1 & 3 & -1 & \ddots & \vdots\\
0 & -1 & 3 & \ddots & 0\\
\vdots & \ddots & \ddots & \ddots & -1\\
0 & \cdots & 0 & -1 & 3
\end{bmatrix}.
\]
Also, let $A_0$ be the empty matrix and define $\det A_0=1$.

When $n=0$, we have
\[
\det A_0=1=F_2,
\]
and when $n=1$, we have
\[
\det A_1=3=F_4.
\]
The following elementary determinant formula will be used repeatedly.
\begin{lemma}
\label{lem:det-A}
For every integer $n\ge 0$,
\[
\det A_n=F_{2n+2}
\]
holds.
\end{lemma}

\begin{proof}
When $n=0$, we have $\det A_0=1=F_2$, and when $n=1$, we have
$\det A_1=3=F_4$.

Let $n\ge2$.
Expanding the determinant of $A_n$ along the last row, we obtain
\[
\det A_n
=
3\det A_{n-1}-\det A_{n-2}.
\]

The even-indexed Fibonacci numbers satisfy
\[
F_{2n+2}
=
3F_{2n}-F_{2n-2}.
\]
Hence the two sequences $\{\det A_n\}_{n\ge 0}$ and $\{F_{2n+2}\}_{n\ge 0}$ have the same recurrence relation and the same initial values.
Therefore,
\[
\det A_n=F_{2n+2}
\]
holds.
\end{proof}

Next, we define the matrix $B_d$ which will be used later.
For $d\ge 2$, put
\[
B_d=
\begin{bmatrix}
d & -\mathbf{1}_{d-1}^{T}\\
-\mathbf{1}_{d-1} & A_{d-1}
\end{bmatrix}.
\]
Here $\mathbf{1}_{d-1}$ denotes the column vector of length $d-1$
all of whose entries are equal to $1$.
For $d=1$, we define $B_1=(1)$.

\begin{lemma}
\label{lem:det-B}
For every integer $d\ge 1$,
\[
\det B_d=L_{2d}-2
\]
holds.
\end{lemma}

\begin{proof}
First, let $d\ge 3$.
The wheel graph $W_{d+1}$ has $d$ rim vertices and one central vertex.
If we delete from its Laplacian matrix the row and column corresponding
to one rim vertex, then the remaining $d-1$ rim vertices form a path-type
part, and the resulting principal submatrix coincides with $B_d$.

Therefore, by Theorem \ref{thm:matrix-tree} and Theorem \ref{thm:wheel-tree},
we have
\[
\det B_d
=
\tau(W_{d+1})
=
L_{2d}-2.
\]

When $d=1$, we have $\det B_1=1=L_2-2$.
When $d=2$, we have
\[
B_2=
\begin{bmatrix}
2 & -1\\
-1 & 3
\end{bmatrix}
\]
and hence
\[
\det B_2
=
6-1
=
5
=
L_4-2.
\]
Thus the assertion holds for every integer $d\ge 1$.
\end{proof}

We rewrite the preceding lemma in a form adapted to the Schur complement.
For $d\ge 1$, define $\beta_d$ by
\[
\beta_d=
\begin{cases}
0,
& d=1,\\[2mm]
\mathbf{1}_{d-1}^{T}
A_{d-1}^{-1}
\mathbf{1}_{d-1},
& d\ge 2
\end{cases}.
\]

\begin{lemma}
\label{lem:schur-block}
For every integer $d\ge 1$,
\[
d-\beta_d
=
\frac{L_{2d}-2}{F_{2d}}
\]
holds.
\end{lemma}

\begin{proof}
When $d=1$, we have $\beta_1=0$, and hence $d-\beta_d=1$.
Therefore,
\[
d-\beta_1
=
1
=
\frac{3-2}{1}
=
\frac{L_2-2}{F_2},
\]
so the assertion holds.

Let $d\ge 2$.
By Lemma \ref{lem:det-A}, we have $\det A_{d-1}=F_{2d}>0$.
Thus $A_{d-1}$ is invertible.
Using the Schur complement with respect to $A_{d-1}$, we obtain
\[
\det B_d
=
\det A_{d-1}
\left(
d-
\mathbf{1}_{d-1}^{T}
A_{d-1}^{-1}
\mathbf{1}_{d-1}
\right).
\]
That is,
\[
\det B_d
=
\det A_{d-1}(d-\beta_d).
\]

Therefore, by Lemma \ref{lem:det-A} and Lemma \ref{lem:det-B}, we get
\[
\det A_{d-1}(d-\beta_d)
=
F_{2d}(d-\beta_d)
=
L_{2d}-2
=
\det B_d.
\]
Hence
\[
d-\beta_d
=
\frac{L_{2d}-2}{F_{2d}}
\]
holds.
\end{proof}

\begin{proof}[Proof of Theorem \ref{thm:rim-roots}]
By Theorem \ref{thm:all-minors}, $\mathcal F_{W_{N+1}}(R)$ is equal to
the determinant of the principal submatrix obtained from the Laplacian
matrix of $W_{N+1}$ by deleting all rows and columns corresponding to
vertices in $R$.

After deleting the root vertices $v_{a_1},v_{a_2},\ldots,v_{a_r}$,
the rim is divided into $r$ mutually disjoint path-type parts.
The cyclic interval $d_i$ contains $d_i-1$ rim vertices which do not belong
to the root set and lie between $v_{a_i}$ and $v_{a_{i+1}}$.
The principal submatrix corresponding to this part is $A_{d_i-1}$.
Here, even after deleting rows and columns from the Laplacian matrix,
the diagonal entries of the remaining vertices are still the degrees
in the original graph.
Hence each diagonal entry is $3$.

When $d_i=1$, the vertices $v_{a_i}$ and $v_{a_{i+1}}$ are adjacent,
and there is no remaining rim vertex between them.
In this case, the corresponding block is interpreted as the empty matrix
$A_0$.

Since the central vertex $c$ is not deleted, the principal submatrix
obtained by deleting the rows and columns corresponding to the root set $R$
has the following form, after omitting empty blocks:
\[
M=
\begin{bmatrix}
N
&
-\mathbf{1}_{d_1-1}^{T}
&
-\mathbf{1}_{d_2-1}^{T}
&
\cdots
&
-\mathbf{1}_{d_r-1}^{T}
\\
-\mathbf{1}_{d_1-1}
&
A_{d_1-1}
&
0
&
\cdots
&
0
\\
-\mathbf{1}_{d_2-1}
&
0
&
A_{d_2-1}
&
\cdots
&
0
\\
\vdots
&
\vdots
&
\vdots
&
\ddots
&
\vdots
\\
-\mathbf{1}_{d_r-1}
&
0
&
0
&
\cdots
&
A_{d_r-1}
\end{bmatrix}.
\]

Put
\[
I=\{i:1\le i\le r,\ d_i\ge 2\}.
\]

First, suppose that $I=\varnothing$.
Then $d_1=\cdots=d_r=1$, and hence $r=N$.
In this case, the principal submatrix is
\[
M=(N).
\]
Therefore,
\[
\det M=N
=
\sum_{i=1}^{N}(L_2-2)
=
\sum_{i=1}^{N}
(L_{2d_i}-2)
\prod_{\substack{1\le j\le N\\j\ne i}}F_{2d_j},
\]
and the assertion holds.
Hence we may assume that $I\ne\varnothing$.

For each $i\in I$, Lemma \ref{lem:det-A} gives
$\det A_{d_i-1}=F_{2d_i}>0$.
Thus $A_{d_i-1}$ is invertible.
Taking the Schur complement with respect to the direct sum of the nonempty
blocks
\[
\bigoplus_{i\in I}A_{d_i-1},
\]
we obtain
\[
\det M
=
\left(
\prod_{i\in I}\det A_{d_i-1}
\right)
\left(
N-\sum_{i\in I}
\mathbf{1}_{d_i-1}^{T}
A_{d_i-1}^{-1}
\mathbf{1}_{d_i-1}
\right).
\]

When $d_i=1$, we have
\[
\det A_0=1,
\qquad
\beta_1=0.
\]
Thus the above expression can be written as
\[
\det M
=
\left(
\prod_{i=1}^{r}\det A_{d_i-1}
\right)
\left(
N-\sum_{i=1}^{r}\beta_{d_i}
\right).
\]

Since the sum of the cyclic intervals is
\[
N=d_1+d_2+\cdots+d_r,
\]
we have
\[
N-\sum_{i=1}^{r}\beta_{d_i}
=
\sum_{i=1}^{r}(d_i-\beta_{d_i}).
\]

Therefore, by Lemma \ref{lem:det-A} and Lemma \ref{lem:schur-block},
\[
\begin{aligned}
\mathcal F_{W_{N+1}}(R)
&=
\det M\\
&=
\left(
\prod_{i=1}^{r}\det A_{d_i-1}
\right)
\left(
\sum_{i=1}^{r}(d_i-\beta_{d_i})
\right) \\
&=
\left(
\prod_{i=1}^{r}F_{2d_i}
\right)
\left(
\sum_{i=1}^{r}
\frac{L_{2d_i}-2}{F_{2d_i}}
\right)\\
&=
\sum_{i=1}^{r}
\left(L_{2d_i}-2\right)
\prod_{\substack{1\le j\le r\\j\ne i}}
F_{2d_j}.
\end{aligned}
\]

\end{proof}

By Theorem \ref{thm:identification-forest}, we also have
\[
\tau(W_{N+1}/R)
=
\mathcal F_{W_{N+1}}(R).
\]
Thus the formula also gives the number of spanning trees of the quotient
graph.

\begin{remark}
\label{rem:gap-multiset}
The right-hand side of Theorem \ref{thm:rim-roots} is a symmetric expression
in $d_1,\ldots,d_r$.
Therefore, the number of rooted spanning forests does not depend on the
order of the cyclic intervals, but only on their multiset.

In particular, even if two root sets have cyclic interval sequences which
are not cyclic permutations or reversals of each other, the corresponding
numbers of rooted spanning forests are equal whenever the two sequences
have the same multiset.
\end{remark}

When $r=1$, the unique cyclic interval is $d_1=N$.
Also, when there is only one root, a rooted spanning forest is just an
ordinary spanning tree.
Hence
\[
\mathcal F_{W_{N+1}}(\{v_a\})
=
\tau(W_{N+1}).
\]
Thus Theorem \ref{thm:wheel-tree} follows from Theorem \ref{thm:rim-roots}.

\begin{corollary}
\label{cor:r-one}
When $r=1$,
\[
\mathcal F_{W_{N+1}}(\{v_a\})
=
\tau(W_{N+1})
=
L_{2N}-2.
\]
\end{corollary}

\begin{proof}
When $r=1$, the unique cyclic interval is
\[
d_1=N.
\]
Therefore, by Theorem \ref{thm:rim-roots}, we obtain
\[
\mathcal F_{W_{N+1}}(\{v_a\})
=
L_{2N}-2.
\]
Moreover, when there is only one root, a rooted spanning forest is just
an ordinary spanning tree, and hence
\[
\mathcal F_{W_{N+1}}(\{v_a\})
=
\tau(W_{N+1}).
\]
\end{proof}

Next, let $r=2$, and let the two cyclic intervals determined by the two
rim root vertices be $d$ and $N-d$.
That is, in Theorem \ref{thm:rim-roots}, we put
\[
d_1=d,\qquad d_2=N-d.
\]
Then we obtain the following corollary.

\begin{corollary}
\label{cor:r-two}
\[
\begin{aligned}
\mathcal F_{W_{N+1}}(\{v_a,v_b\})
&=\tau(W_{N+1}/\{v_a,v_b\}) \\
&=\left(L_{2d}-2\right)F_{2(N-d)}+\left(L_{2(N-d)}-2\right)F_{2d}
\end{aligned}
\]
holds.
\end{corollary}

\begin{remark}
\label{rem:two-rim-previous}
Corollary \ref{cor:r-two} gives the case where two rim vertices are
identified into one vertex.
It agrees with the formula of Tamura and Tanaka
\cite[Theorem 5]{TamuraTanaka2026}.
Therefore, Theorem \ref{thm:rim-roots} can be regarded as an extension of
the known formula for two identified vertices to the case where an arbitrary
number of rim vertices are identified.
\end{remark}

Furthermore, if the rim vertices are chosen as roots so that all cyclic
intervals are equal, then Theorem \ref{thm:rim-roots} gives the following
corollary.

\begin{corollary}
\label{cor:equally-spaced}
Let $N=rq$, and choose $r$ rim vertices at equal intervals.
That is, assume that
\[
d_1=d_2=\cdots=d_r=q.
\]
Then
\[
\mathcal F_{W_{N+1}}(R)
=\tau(W_{N+1}/R)
=r\left(L_{2q}-2\right)F_{2q}^{\,r-1}
\]
holds.
\end{corollary}

\begin{proof}
By applying Theorem \ref{thm:rim-roots} under the assumption that all cyclic
intervals are equal to $q$, we obtain
\[
\begin{aligned}
\mathcal F_{W_{N+1}}(R)
&=
\sum_{i=1}^{r}
\left(L_{2q}-2\right)
F_{2q}^{\,r-1}\\
&=
r\left(L_{2q}-2\right)
F_{2q}^{\,r-1}.
\end{aligned}
\]
The formula for the number of spanning trees of the quotient graph follows
from Theorem \ref{thm:identification-forest}.
\end{proof}

\subsection{The case where the central vertex is included in the root set}
\label{sec:center-root}
When the central vertex is included in the root set, we obtain a product
formula which is even simpler than Theorem \ref{thm:rim-roots}.
Indeed, if the central vertex is included in the root set, the corresponding
principal submatrix of the Laplacian matrix decomposes into a direct sum
of blocks corresponding to the cyclic intervals on the rim.

\begin{proof}[Proof of Theorem \ref{thm:center-root}]
By Theorem \ref{thm:all-minors},
$\mathcal F_{W_{N+1}}(\{c\}\cup R)$ is equal to the determinant of the
principal submatrix obtained from the Laplacian matrix of $W_{N+1}$ by
deleting all rows and columns corresponding to $c$ and to the vertices
in $R$.

After deleting the row and column corresponding to the central vertex $c$,
the remaining off-diagonal entries correspond only to adjacencies on the rim.
Furthermore, after deleting the rows and columns corresponding to the root
vertices $v_{a_1},v_{a_2},\ldots,v_{a_r}$, the remaining rim vertices are
divided into $r$ mutually disjoint path-type parts.

The cyclic interval $d_i$ contains $d_i-1$ rim vertices which do not belong
to the root set and lie between $v_{a_i}$ and $v_{a_{i+1}}$.
In a principal submatrix of the Laplacian matrix, the diagonal entries of
the remaining vertices are still the degrees in the original graph.
Since each rim vertex of a wheel graph has degree $3$, and the off-diagonal
entries corresponding to adjacent rim vertices are $-1$, the block
corresponding to the cyclic interval $d_i$ coincides with $A_{d_i-1}$.
Therefore, the corresponding principal submatrix is
\[
A_{d_1-1}
\oplus
A_{d_2-1}
\oplus
\cdots
\oplus
A_{d_r-1}.
\]
When $d_i=1$, there is no remaining rim vertex between $v_{a_i}$ and
$v_{a_{i+1}}$, and the corresponding block is interpreted as the empty
matrix $A_0$.
In this case, $\det A_0=1$.

Hence, by Lemma \ref{lem:det-A}, we obtain
\[
\begin{aligned}
\mathcal F_{W_{N+1}}(\{c\}\cup R)
&=
\det
\left(
A_{d_1-1}
\oplus
A_{d_2-1}
\oplus
\cdots
\oplus
A_{d_r-1}
\right)\\
&=
\prod_{i=1}^{r}\det A_{d_i-1}\\
&=
\prod_{i=1}^{r}F_{2d_i}.
\end{aligned}
\]

Finally, by Theorem \ref{thm:identification-forest}, we have
\[
\tau\bigl(W_{N+1}/(\{c\}\cup R)\bigr)
=
\mathcal F_{W_{N+1}}(\{c\}\cup R).
\]
Thus the formula for the number of spanning trees of the quotient graph
also follows.

\end{proof}

\begin{remark}
\label{rem:rim-center-relation}
Comparing Theorem \ref{thm:rim-roots} and Theorem \ref{thm:center-root},
we obtain
\[
\mathcal F_{W_{N+1}}(R)
=
\mathcal F_{W_{N+1}}(\{c\}\cup R)
\sum_{i=1}^{r}
\frac{L_{2d_i}-2}{F_{2d_i}}.
\]
That is, the product
\[
\prod_{i=1}^{r}F_{2d_i}
\]
which appears when the central vertex is included in the root set also
appears as a common factor in the formula for the case where the central
vertex is not included.
The latter is obtained by multiplying this product by the sum of the
contributions
\[
\frac{L_{2d_i}-2}{F_{2d_i}}
\]
from the cyclic intervals.
\end{remark}

When $r=1$, the unique cyclic interval is $d_1=N$.
Thus Theorem \ref{thm:center-root} gives the following corollary.

\begin{corollary}
\label{cor:center-one-rim}
When $r=1$,
\[
\mathcal F_{W_{N+1}}(\{c,v_a\})
=\tau(W_{N+1}/\{c,v_a\})
=F_{2N}
\]
holds.
\end{corollary}

Next, let $r=2$, and let the two cyclic intervals determined by the two
rim root vertices be $d$ and $N-d$.
That is, in Theorem \ref{thm:center-root}, we put
\[
d_1=d,\qquad d_2=N-d.
\]
Then we obtain the following corollary.

\begin{corollary}
\label{cor:center-two-rim}
\[
\mathcal F_{W_{N+1}}(\{c,v_a,v_b\})
=\tau(W_{N+1}/\{c,v_a,v_b\})
=F_{2d}F_{2(N-d)}
\]
holds.
\end{corollary}

Furthermore, if the rim vertices are chosen as roots so that all cyclic
intervals are equal, then Theorem \ref{thm:center-root} gives the following
corollary.

\begin{corollary}
\label{cor:center-equally-spaced}
Let $N=rq$, and choose $r$ rim vertices at equal intervals.
That is, assume that
\[
d_1=d_2=\cdots=d_r=q.
\]
Then
\[
\mathcal F_{W_{N+1}}(\{c\}\cup R)
=
\tau\bigl(W_{N+1}/(\{c\}\cup R)\bigr)
=
F_{2q}^{\,r}
\]
holds.
\end{corollary}

\begin{proof}
By Theorem \ref{thm:center-root}, we have
\[
\mathcal F_{W_{N+1}}(\{c\}\cup R)
=
\prod_{i=1}^{r}F_{2q}
=
F_{2q}^{\,r}.
\]
The formula for the number of spanning trees of the quotient graph follows
from Theorem \ref{thm:identification-forest}.
\end{proof}

\begin{remark}
\label{rem:center-fan}
The product formula in Theorem \ref{thm:center-root} shows that, when the
central vertex is included in the root set, each cyclic interval on the rim
contributes independently to the determinant.

Indeed, the contribution from the $i$-th interval whose cyclic length is
$d_i$ is
\[
\det A_{d_i-1}=F_{2d_i}.
\]
On the other hand, by Theorem \ref{thm:fan-tree}, we have
\[
\tau(\operatorname{Fan}_{d_i+1})
=
F_{2d_i}.
\]
Therefore, the factor arising from each cyclic interval agrees with the
number of spanning trees of the fan graph $\operatorname{Fan}_{d_i+1}$.
In this sense, the number of rooted spanning forests when the central
vertex is included in the root set can be understood as the product of
fan-graph-type contributions corresponding to the cyclic intervals.
\end{remark}

\subsection{Matchings and rooted spanning forests}
\label{sec:matching}

The fan-graph interpretation in the preceding subsection can be further
rewritten in terms of graph matchings.
Let $m(P_n)$ denote the number of matchings of the path graph $P_n$.
From the bijection of Benjamin and Yerger \cite{BY2006} between matchings
of $P_{2n-1}$ and spanning trees of $\operatorname{Fan}_{n+1}$, we have
\[
m(P_{2n-1})
=
\tau(\operatorname{Fan}_{n+1})
=
F_{2n}.
\]
Therefore, by Theorem \ref{thm:center-root}, we obtain the following
corollary.

\begin{corollary}
\label{cor:matching-center}
\[
\mathcal F_{W_{N+1}}(\{c\}\cup R)
=
\prod_{i=1}^{r}m(P_{2d_i-1}).
\]
\end{corollary}

On the other hand, since
\[
m^*(C_{2d})=L_{2d}-2,
\]
Theorem \ref{thm:rim-roots} gives the following corollary.

\begin{corollary}
\label{cor:matching-rim}
\[
\mathcal F_{W_{N+1}}(R)
=
\sum_{i=1}^{r}
m^*(C_{2d_i})
\prod_{\substack{1\le j\le r\\j\ne i}}
m(P_{2d_j-1}).
\]
\end{corollary}

Thus Theorem \ref{thm:center-root} and Theorem \ref{thm:rim-roots} have
combinatorial representations in terms of the numbers of matchings of
path graphs and cycle graphs.

\section{Summation formulas over root sets}
\label{sec:sum-formulas}

In the previous section, we gave explicit formulas for rooted spanning
forests for an arbitrary fixed root set.

In this section, for the wheel graph $W_{N+1}$, we fix the number
$r$ of rim root vertices, where $1\le r\le N$, and compute the sum of
the numbers of rooted spanning forests over all rim root sets of size $r$.

Let
\[
V_{\mathrm{rim}}
=
\{v_1,v_2,\ldots,v_N\}
\]
be the set of rim vertices.
For a formal power series $f(z)$, we denote by $[z^n]f(z)$ the coefficient
of $z^n$ in $f(z)$.

We define $A_{N,r}$ to be the sum of the numbers of rooted spanning forests
with $r$ rim vertices as roots, taken over all rim root sets of size $r$.
We also define $B_{N,r}$ to be the sum of the numbers of rooted spanning
forests with the central vertex $c$ and $r$ rim vertices as roots, taken
over all rim root sets of size $r$.
That is,
\[
A_{N,r}
=\sum_{\substack{R\subseteq V_{\mathrm{rim}}\\|R|=r}}\mathcal F_{W_{N+1}}(R), \qquad
B_{N,r}
=\sum_{\substack{R\subseteq V_{\mathrm{rim}}\\|R|=r}}\mathcal F_{W_{N+1}}(\{c\}\cup R).
\]

First, we compute two generating functions which will be used in the proof
of Theorem \ref{thm:sum-formulas}.
\begin{lemma}
\label{lem:generating-functions}

\[
\sum_{d\ge 1}F_{2d}z^d
=
\frac{z}{1-3z+z^2}, \qquad
\sum_{d\ge 1}(L_{2d}-2)z^d
=\frac{z(1+z)}{(1-z)(1-3z+z^2)}.
\]
\end{lemma}

\begin{proof}
The even-indexed Fibonacci numbers satisfy
\[
F_{2d}
=
3F_{2d-2}-F_{2d-4}
\qquad
(d\ge 3).
\]
Here $F_2=1$ and $F_4=3$.
Thus, by a standard generating-function computation, we obtain
\[
\sum_{d\ge 1}F_{2d}z^d
=
\frac{z}{1-3z+z^2}.
\]

Similarly, the even-indexed Lucas numbers satisfy
\[
L_{2d}
=
3L_{2d-2}-L_{2d-4}
\qquad
(d\ge 3).
\]
Here $L_2=3$ and $L_4=7$.

Therefore,
\[
\sum_{d\ge 1}L_{2d}z^d
=
\frac{3z-2z^2}{1-3z+z^2}
\]
holds.
Since
\[
\sum_{d\ge 1}2z^d
=
\frac{2z}{1-z},
\]
we have
\[
\begin{aligned}
\sum_{d\ge 1}(L_{2d}-2)z^d
&=
\frac{3z-2z^2}{1-3z+z^2}
-
\frac{2z}{1-z}\\
&=
\frac{z(1+z)}
{(1-z)(1-3z+z^2)}.
\end{aligned}
\]
\end{proof}

\begin{theorem}
\label{thm:sum-formulas}
For every integer $1\le r\le N$,
\[
A_{N,r}
=N[z^{N-r}]\frac{1+z}{(1-z)(1-3z+z^2)^r}, \qquad
B_{N,r}
=\frac{N}{r}[z^{N-r}]\frac{1}{(1-3z+z^2)^r}
\]
holds.
\end{theorem}

\begin{proof}
Put
\[
G(z)
=
\sum_{d\ge 1}F_{2d}z^d,
\qquad
H(z)
=
\sum_{d\ge 1}(L_{2d}-2)z^d.
\]
Then, by Lemma \ref{lem:generating-functions},
\[
G(z)
=
\frac{z}{1-3z+z^2},
\qquad
H(z)
=
\frac{z(1+z)}
{(1-z)(1-3z+z^2)}.
\]

Take a root set $R=\{v_{a_1},v_{a_2},\ldots,v_{a_r}\}$ consisting of
rim vertices, and let $d_1,d_2,\ldots,d_r$ be the corresponding cyclic
intervals.
Then
\[
d_1+d_2+\cdots+d_r=N,
\qquad
d_i\ge 1.
\]

We consider ordered $r$-tuples $(d_1,d_2,\ldots,d_r)$ of positive integers
satisfying $d_1+d_2+\cdots+d_r=N$.
If such a cyclic interval sequence and an initial rim vertex are chosen,
then the rim root set with that initial vertex as a root is uniquely
determined.
Conversely, if a rim root set and one of its root vertices as the initial
vertex are chosen, then the corresponding ordered cyclic interval sequence
is uniquely determined.

There are $N$ choices for the initial rim vertex.
On the other hand, each root set is counted exactly $r$ times, according
to the choice of one of its $r$ root vertices as the initial vertex.
Therefore, by Theorem \ref{thm:rim-roots}, we obtain
\[
\begin{aligned}
A_{N,r}
&=
\frac{N}{r}
\sum_{\substack{d_1+\cdots+d_r=N\\d_i\ge 1}}
\sum_{i=1}^{r}
(L_{2d_i}-2)
\prod_{\substack{1\le j\le r\\j\ne i}}
F_{2d_j} \\
&=
\frac{N}{r}
[z^N]\,rH(z)G(z)^{r-1}\\
&=
N[z^N]H(z)G(z)^{r-1}\\
&=
N[z^N]
\frac{z(1+z)}
{(1-z)(1-3z+z^2)}
\left(
\frac{z}{1-3z+z^2}
\right)^{r-1}\\
&=
N[z^N]
\frac{z^r(1+z)}
{(1-z)(1-3z+z^2)^r}\\
&=
N[z^{N-r}]
\frac{1+z}
{(1-z)(1-3z+z^2)^r}.
\end{aligned}
\]

Similarly, taking into account the multiplicity coming from the choice
of the initial root vertex, Theorem \ref{thm:center-root} gives
\[
\begin{aligned}
B_{N,r}
&=
\frac{N}{r}
\sum_{\substack{d_1+\cdots+d_r=N\\d_i\ge 1}}
\prod_{i=1}^{r}F_{2d_i}\\
&=
\frac{N}{r}
[z^N]G(z)^r\\
&=
\frac{N}{r}
[z^N]
\left(
\frac{z}{1-3z+z^2}
\right)^r\\
&=
\frac{N}{r}
[z^{N-r}]
\frac{1}
{(1-3z+z^2)^r}.
\end{aligned}
\]

Thus the two formulas are obtained.
\end{proof}

\begin{remark}
\label{rem:sum-special-cases}
If $r=1$ in Theorem \ref{thm:sum-formulas}, then we obtain
\[
A_{N,1}
=
N(L_{2N}-2),
\qquad
B_{N,1}
=
NF_{2N}.
\]
Indeed, there are $N$ choices for the rim root vertex, and the number of
rooted spanning forests for each root is given by Theorem
\ref{thm:wheel-tree} and Corollary \ref{cor:center-one-rim}, respectively.

If $r=N$, then
\[
A_{N,N}=N,
\qquad
B_{N,N}=1.
\]
Indeed, when all rim vertices are roots, two distinct rim root vertices
cannot belong to the same connected component.
Therefore, we must choose exactly one spoke joining the central vertex $c$
to one of the rim root vertices, and there are $N$ choices.

On the other hand, if the central vertex is also included in the root set,
then all vertices are roots, and hence the only rooted spanning forest is
the forest with no edges.
Thus $B_{N,N}=1$.
\end{remark}

\section{Extremal problems for root configurations}
\label{sec:extremal}

In this section, we study the relation between the arrangement of the root
set and the number of rooted spanning forests.
By Theorem \ref{thm:rim-roots} and Theorem \ref{thm:center-root}, the number
of rooted spanning forests is determined not by the exact positions of the
root vertices, but only by the multiset of cyclic intervals between adjacent
root vertices.
Therefore, for fixed $N$ and a fixed number $r$ of rim root vertices, we
determine which root configurations maximize or minimize the number of
rooted spanning forests.

In what follows, let $1\le r\le N$, and write $N=qr+s$ with $0\le s<r$.
That is, $q=\left\lfloor\frac{N}{r}\right\rfloor$.

We prepare a two-variable smoothing inequality which will be used in the
proofs of Theorem \ref{thm:center-extremal} and Theorem
\ref{thm:rim-extremal}.
For $n\ge 1$, put
\[
f_n=F_{2n},
\qquad
g_n=L_{2n}-2.
\]

\begin{lemma}
\label{lem:extremal-smoothing}
Let $a,b$ be integers satisfying $1\le a\le b-2$.
Then
\[
\begin{aligned}
f_{a+1}f_{b-1}&>f_af_b, \\
g_{a+1}f_{b-1}
+
g_{b-1}f_{a+1}
&>
g_af_b+g_bf_a,
\end{aligned}
\]
hold.
\end{lemma}

\begin{proof}
Put
\[
\varphi=\frac{1+\sqrt5}{2},
\qquad
\lambda=\varphi^2.
\]
Then $\lambda>1$, and by Binet's formula, we have
\[
f_n
=
F_{2n}
=
\frac{\lambda^n-\lambda^{-n}}{\sqrt5}
=
\frac{(\lambda^n-1)(\lambda^n+1)}
{\sqrt5\,\lambda^n}
\]
and
\[
g_n
=
L_{2n}-2
=
\lambda^n+\lambda^{-n}-2
=
\frac{(\lambda^n-1)^2}{\lambda^n}.
\]

First,
\[
f_af_b
=
\frac{(\lambda^{2a}-1)(\lambda^{2b}-1)}
{5\lambda^{a+b}}.
\]
The sum $a+b$ is unchanged by the smoothing operation
\[
(a,b)\longmapsto(a+1,b-1).
\]
Therefore,
\[
\begin{aligned}
f_{a+1}f_{b-1}-f_af_b
&=
\frac{\lambda^2-1}
{5\lambda^{a+b}}
\left(
\lambda^{2b-2}-\lambda^{2a}
\right).
\end{aligned}
\]
Since $b\ge a+2$, we have $2b-2>2a$.
Hence
\[
f_{a+1}f_{b-1}-f_af_b>0.
\]

Next, using the above expressions and simplifying, we get
\[
g_af_b+g_bf_a
=
\frac{2}{\sqrt5}
\frac{
(\lambda^a-1)(\lambda^b-1)(\lambda^{a+b}-1)
}
{\lambda^{a+b}}.
\]
Thus
\[
g_{a+1}f_{b-1}
+
g_{b-1}f_{a+1}
-
g_af_b
-
g_bf_a
=
\frac{
2(\lambda-1)(\lambda^{a+b}-1)
(\lambda^{b-1}-\lambda^a)
}
{\sqrt5\,\lambda^{a+b}}.
\]
Since $b\ge a+2$, we have $b-1>a$, and hence
$\lambda^{b-1}-\lambda^a>0$.
Therefore,
\[
g_{a+1}f_{b-1}
+
g_{b-1}f_{a+1}
>
g_af_b+g_bf_a.
\]
\end{proof}

\subsection{The case where the central vertex is included in the root set}

We first consider the case where the central vertex is included in the
root set.

\begin{theorem}
\label{thm:center-extremal}
Let $1\le r\le N$, and write
\[
N=qr+s,
\qquad
0\le s<r.
\]
Let $R$ be a set consisting of $r$ rim vertices, and let
$d_1,d_2,\ldots,d_r$ be the corresponding cyclic intervals.
The number $\mathcal F_{W_{N+1}}(\{c\}\cup R)$ of rooted spanning forests
with the central vertex included in the root set is maximized if and only if
the cyclic intervals are as equal as possible.
That is, the maximum is attained when, in the cyclic interval sequence
$(d_1,\ldots,d_r)$, the value $q+1$ appears $s$ times and the value $q$
appears $r-s$ times.
The maximum value is
\[
F_{2(q+1)}^{\,s}F_{2q}^{\,r-s}.
\]

On the other hand, the minimum is attained if and only if the cyclic interval
sequence $(d_1,\ldots,d_r)$ is a permutation of
\[
(N-r+1,1,\ldots,1).
\]
The minimum value is
\[
F_{2(N-r+1)}.
\]
\end{theorem}

\begin{proof}
The cyclic intervals $d_1,d_2,\ldots,d_r$ satisfy
$d_1+d_2+\cdots+d_r=N$.

If $r=1$, then the assertion about the maximum and minimum is clear from
Corollary \ref{cor:center-one-rim}.

Assume that $r\ge 2$.
By Theorem \ref{thm:center-root},
\[
\mathcal F_{W_{N+1}}(\{c\}\cup R)
=
\prod_{i=1}^{r}f_{d_i}.
\]

First, suppose that two cyclic intervals $a,b$ satisfy $b\ge a+2$.
If we replace them by
\[
(a,b)\longmapsto(a+1,b-1),
\]
then the total sum of the cyclic intervals does not change.
By Lemma \ref{lem:extremal-smoothing}, we have
$f_{a+1}f_{b-1}>f_af_b$.
Hence, with all other factors fixed, this smoothing operation strictly
increases the number of rooted spanning forests.

We repeat this operation as long as there are two cyclic intervals whose
difference is at least $2$.
Each smoothing operation changes the sum of squares of the two cyclic
intervals from $a^2+b^2$ to $(a+1)^2+(b-1)^2$, and the difference is
\[
a^2+b^2-\bigl((a+1)^2+(b-1)^2\bigr)
=
2(b-a-1)>0.
\]
Thus the total sum of squares of the cyclic intervals decreases at each
step, and the process terminates after finitely many steps.
At the end of the process, any two cyclic intervals differ by at most $1$.
Since the sum of the cyclic intervals is $N=qr+s$, each cyclic interval is
$q$ or $q+1$, with $q+1$ appearing $s$ times and $q$ appearing $r-s$ times.
Since the value strictly increases at each smoothing step, this condition
is necessary and sufficient for the maximum.
Therefore, the maximum value is
\[
F_{2(q+1)}^{\,s}F_{2q}^{\,r-s}.
\]

Next, we consider the minimum.
Suppose that two cyclic intervals $a,b$ satisfy $2\le a\le b$.
Applying Lemma \ref{lem:extremal-smoothing} to $a-1$ and $b+1$, we obtain
$f_af_b>f_{a-1}f_{b+1}$.
Therefore, the operation
\[
(a,b)\longmapsto(a-1,b+1)
\]
which makes one cyclic interval smaller and the other larger strictly
decreases the number of rooted spanning forests.

We can repeat this operation as long as at least two cyclic intervals are
greater than $1$.
Each operation changes the sum of squares of the two cyclic intervals from
$a^2+b^2$ to $(a-1)^2+(b+1)^2$, and the increase is
\[
(a-1)^2+(b+1)^2-a^2-b^2
=
2(b-a+1)>0.
\]
Since the cyclic intervals are positive integers and their total sum is
fixed to be $N$, this process also terminates after finitely many steps.
At the end, at most one cyclic interval is greater than $1$.
Since the total sum of the cyclic intervals is $N$, the cyclic interval
sequence is a permutation of
\[
(N-r+1,1,1,\ldots,1).
\]
When $N=r$, all cyclic intervals are equal to $1$.
Since the value strictly decreases at each operation, this condition is
necessary and sufficient for the minimum.

Finally, since $F_2=1$, the minimum value is $F_{2(N-r+1)}$.
\end{proof}

\subsection{The case where only rim vertices are roots}

Next, we consider the case where the central vertex is not included in the
root set, and only rim vertices are roots.

\begin{theorem}
\label{thm:rim-extremal}
Let $1\le r\le N$, and write
\[
N=qr+s,
\qquad
0\le s<r.
\]
Let $R$ be a set consisting of $r$ rim vertices, and let
$d_1,d_2,\ldots,d_r$ be the corresponding cyclic intervals.
The number $\mathcal F_{W_{N+1}}(R)$ of rooted spanning forests with only
rim vertices as roots is maximized if and only if the cyclic intervals are
as equal as possible.
That is, the maximum is attained when, in the cyclic interval sequence
$(d_1,\ldots,d_r)$, the value $q+1$ appears $s$ times and the value $q$
appears $r-s$ times.
The maximum value is
\[
\begin{aligned}
F_{2(q+1)}^{\,s}F_{2q}^{\,r-s}\times\left(s\frac{L_{2(q+1)}-2}{F_{2(q+1)}}+(r-s)\frac{L_{2q}-2}{F_{2q}}\right).
\end{aligned}
\]

On the other hand, the minimum is attained if and only if, in the cyclic
interval sequence $(d_1,\ldots,d_r)$, one cyclic interval is $N-r+1$ and the
remaining $r-1$ cyclic intervals are all equal to $1$.
The minimum value is
\[
L_{2(N-r+1)}-2
+
(r-1)F_{2(N-r+1)}.
\]
\end{theorem}

\begin{proof}
The cyclic intervals $d_1,d_2,\ldots,d_r$ satisfy
$d_1+d_2+\cdots+d_r=N$.

If $r=1$, then the assertion about the maximum and minimum is clear from
Theorem \ref{thm:wheel-tree}.

Assume that $r\ge 2$.
For simplicity, write
$\Phi(d_1,\ldots,d_r)=\mathcal F_{W_{N+1}}(R)$.
By Theorem \ref{thm:rim-roots},
\[
\Phi(d_1,\ldots,d_r)
=
\mathcal F_{W_{N+1}}(R)
=
\sum_{i=1}^{r}
g_{d_i}
\prod_{\substack{1\le j\le r\\j\ne i}}
f_{d_j}.
\]

Let two cyclic intervals be $d_p=a$ and $d_q=b$, and assume that
$b\ge a+2$.
Fix all remaining cyclic intervals, and put
\[
P=
\prod_{\substack{1\le k\le r\\k\ne p,q}}
f_{d_k},
\qquad
C=
\sum_{\substack{1\le k\le r\\k\ne p,q}}
\frac{g_{d_k}}{f_{d_k}}.
\]
The empty product is interpreted as $1$, and the empty sum is interpreted
as $0$.

Since each $f_{d_k}$ is positive and each $g_{d_k}$ is nonnegative, we have
$P>0$ and $C\ge 0$.
Then
\[
\Phi(d_1,\ldots,d_r)
=
P
\left(
g_af_b+g_bf_a+Cf_af_b
\right).
\]

After replacing the cyclic intervals $(a,b)$ by $(a+1,b-1)$, the value becomes
\[
P
\left(
g_{a+1}f_{b-1}
+
g_{b-1}f_{a+1}
+
Cf_{a+1}f_{b-1}
\right).
\]
By Lemma \ref{lem:extremal-smoothing} and $P>0,C\ge 0$, we have
\[
g_{a+1}f_{b-1}
+
g_{b-1}f_{a+1}
>
g_af_b+g_bf_a
\]
and
\[
f_{a+1}f_{b-1}>f_af_b.
\]
Moreover, $P>0$, $C\ge 0$. 

Therefore,
\[
\Phi(\ldots,a+1,\ldots,b-1,\ldots)
>
\Phi(\ldots,a,\ldots,b,\ldots).
\]

Thus smoothing two cyclic intervals whose difference is at least $2$
strictly increases the number of rooted spanning forests.
As in the case where the central vertex is included in the root set, this
operation terminates after finitely many steps, and at the end all cyclic
intervals differ by at most $1$.
Therefore, the necessary and sufficient condition for attaining the maximum
is that, among the cyclic intervals, $q+1$ appears $s$ times and $q$ appears
$r-s$ times.
Using the product-and-sum form in Theorem \ref{thm:rim-roots}, the maximum
value is
\[
\begin{aligned}
F_{2(q+1)}^{\,s}F_{2q}^{\,r-s}\times\left(s\frac{L_{2(q+1)}-2}{F_{2(q+1)}}+(r-s)\frac{L_{2q}-2}{F_{2q}}\right).
\end{aligned}
\]

Next, we consider the minimum.
Suppose that two cyclic intervals satisfy $2\le a\le b$.
Applying the smoothing inequality proved above to $a-1$ and $b+1$, we get
\[
\Phi(\ldots,a-1,\ldots,b+1,\ldots)
<
\Phi(\ldots,a,\ldots,b,\ldots).
\]
Therefore, if we take $1$ from one cyclic interval and add it to a larger
cyclic interval, then the number of rooted spanning forests strictly
decreases.

Repeating this operation, we reach a configuration in which at most one
cyclic interval is greater than $1$.
Since the total sum of the cyclic intervals is $N$, the necessary and
sufficient condition for attaining the minimum is that the cyclic interval
sequence is a permutation of
\[
(N-r+1,1,1,\ldots,1).
\]
When $N=r$, all cyclic intervals are equal to $1$.

Put $m=N-r+1$.
For the minimizing configuration, the cyclic interval sequence is a
permutation of
\[
(m,1,\ldots,1).
\]
Since $F_2=1$ and $L_2-2=1$, we obtain
\[
\begin{aligned}
\mathcal F_{W_{N+1}}(R)
&=(L_{2m}-2)F_2^{\,r-1}+(r-1)(L_2-2)F_{2m}F_2^{\,r-2} \\
&=L_{2m}-2+(r-1)F_{2m}.
\end{aligned}
\]
Thus the minimum value is
\[
L_{2(N-r+1)}-2
+
(r-1)F_{2(N-r+1)}.
\]
\end{proof}

\section{Conclusion and future problems}
\label{sec:conclusion}
In this paper, we studied rooted spanning forests in the wheel graph
$W_{N+1}$ and gave explicit formulas for their numbers in terms of
Fibonacci and Lucas numbers.
These formulas do not depend on the order of the cyclic intervals, but only
on their multiset.
Moreover, by using the correspondence between rooted spanning forests and
vertex identification, we also obtained formulas for the number of spanning
trees of quotient graphs of wheel graphs obtained by identifying several
vertices into one vertex.
In particular, the formula for the case where only rim vertices are roots
extends the known result for two identified vertices to an arbitrary number
of rim vertices.

We also fixed the number of rim root vertices and expressed the sum of the
numbers of rooted spanning forests over all root sets as a coefficient of a
generating function.
Furthermore, we solved the extremal problem for the arrangement of roots.
We showed that, whether or not the central vertex is included in the root
set, the number of rooted spanning forests is maximized when the rim root
vertices are placed as evenly as possible, and minimized when they are placed
as consecutively as possible.

For future work, it would be interesting to consider block-type vertex
identification, in which several vertex sets are identified into distinct
vertices, rather than identifying all root vertices into a single vertex.
In this case, by using the block structure of Laplacian principal submatrices
and transfer matrices, one may try to derive explicit formulas for the number
of spanning trees and rooted spanning forests in terms of Chebyshev
polynomials or generalized Lucas-type sequences.

\section*{Acknowledgments}

The research of Yuuho Tanaka was supported by
JSPS KAKENHI Grant Number JP25K23339.

\section*{Conflict of interest}

The authors declare that they have no conflict of interest.

\section*{Data availability}

Data sharing is not applicable to this article as no datasets were generated or analysed during the current study.



\begin{thebibliography}{99}

\bibitem{BY2006}
A. T. Benjamin and C. R. Yerger,
Combinatorial interpretations of spanning tree identities,
\textit{Bulletin of the Institute of Combinatorics and its Applications}
\textbf{47} (2006), 37--42.

\bibitem{Chaiken1982}
S. Chaiken,
A combinatorial proof of the all minors matrix tree theorem,
\textit{SIAM Journal on Algebraic and Discrete Methods}
\textbf{3} (1982), 319--329.

\bibitem{Chebotarev2008}
P. Chebotarev,
Spanning forests and the golden ratio,
\textit{Discrete Applied Mathematics}
\textbf{156} (2008), 813--821.

\bibitem{ChebotarevShamis1997}
P. Chebotarev and E. Shamis,
The matrix-forest theorem and measuring relations in small social groups,
\textit{Automation and Remote Control}
\textbf{58} (1997), 1505--1514.

\bibitem{Hilton1974}
A. J. W. Hilton,
Spanning trees and Fibonacci and Lucas numbers,
\textit{The Fibonacci Quarterly}
\textbf{12} (1974), 259--262.

\bibitem{MiezakiTamura2025}
T. Miezaki and S. Tamura,
Fibonacci and Lucas numbers arising from two-component spanning forests of wheel graphs,
arXiv:2512.18214, 2025.

\bibitem{Myers1971}
B. R. Myers,
Number of spanning trees in a wheel,
\textit{IEEE Transactions on Circuit Theory}
\textbf{CT-18} (1971), 280--282.

\bibitem{Sedlacek1969}
J. Sedl\'a\v{c}ek,
On the number of spanning trees of finite graphs,
\textit{\v{C}asopis pro p\v{e}stov\'an\'i matematiky}
\textbf{94} (1969), 217--221.

\bibitem{TamuraTanaka2026}
S. Tamura and Y. Tanaka,
Number of spanning trees in a wheel graph with two identified vertices
via hitting times,
\textit{Discrete Applied Mathematics}
\textbf{387} (2026), 318--336.

\end{thebibliography}
\end{document}